\documentclass{article}
\usepackage{amssymb,amsmath,amsthm}
\DeclareMathAlphabet\mathsfsl{OML}{cmm}{b}{it}
\newcommand\cl{\mathop{\mathsfsl{cl}}}

\newcommand\M{\mathcal M}
\newcommand\N{\mathcal N}

\newcommand\LL{\mathcal L}
\newcommand\C[1]{{\upshape(Z\ifx*#1\relax
$2^*$\else #1\fi)}}
\newcommand\CF{\phantom{$^*$}}
\newcommand\restr{\mathord\upharpoonright}

\newtheorem{theorem}{Theorem}
\newtheorem{lemma}[theorem]{Lemma}
\newtheorem{claim}[theorem]{Claim}

\newtheorem{corollary}[theorem]{Corollary}
\theoremstyle{remark}
\newtheorem{remark}{Remark}
\theoremstyle{plain}

\let\theta\vartheta

\title{\bf Cyclic flats of a polymatroid}
\author{Laszlo Csirmaz%
\thanks{Central European University}
\thanks{emal:~csirmaz@renyi.hu}}
\date{\small\it Dedicated to the memory of Frantisek Mat\'u\v s}

\begin{document}
\maketitle
\begin{abstract}

Polymatroids can be considered as ``fractional matroids'' where the rank
function is not required to be integer valued. Many, but not every notion in
matroid terminology translates naturally to polymatroids. Defining cyclic
flats of a polymatroid carefully, the characterization by Bonin and de Mier of
the ranked lattice of cyclic flats carries over to
polymatroids. The main tool, which might be of independent interest, is a
convolution-like method which creates a polymatroid from a ranked lattice
and a discrete measure. Examples show the ease of using the convolution
technique.



\end{abstract}


\section{Introduction}\label{sec:intro}

Cyclic flats of a matroid played an important role in matroid theory. They
form a \emph{ranked lattice}, i.e., a lattice with a non-negative number
assigned to lattice elements. Bonin and de Mier in \cite{cyclic-flats} gave
a characterization of the ranked lattices arising this way. They quote Sims 
\cite{sims} where the rank function of the embedding matroid is specified
explicitly by some convolution-like formula.

Generalizing cyclic flats to polymatroids is not completely straightforward,
see \cite{pcyclic} and \cite{distr} where the same definition of
polymatroidal cyclic flats arose as we use in this paper. Our main
contribution is a complete characterization of the ranked lattice of cyclic
flats of a polymatroid \emph{together with singleton ranks}. Motivated by
Sims' construction \cite{sims}, the \emph{convolution} of a ranked lattice
and a discrete measure (determined by the singleton ranks) is defined. This
definition is an extension of the usual convolution of polymatroids, for
details see \cite{matus-csirmaz}. Similarly to the matroid case, this
convolution recovers the embedding polymatroid for the given ranked lattice
and measure. We carefully identify the role of different conditions. Knowing
which condition ensures what property allows us to use the convolution
to create polymatroids with desired properties. It is illustrated by a
simple example. A more substantial application is in \cite{sticky-poly}.

Traditionally, convolution is symmetric. In the last section we propose the
convolution of two ranked lattices which is symmetric and falls back to the
previous notion when the second lattice is a complete modular polymatroid.
In a special case this convolution gives an interesting polymatroid
extension. It is open under which general and useful conditions will the
convolution be a polymatroid.

\section{Preliminaries}\label{sec:prelim}

\subsection{Notation}
All sets in this paper are finite. A \emph{polymatroid $\M=(f,M)$} is a
non-negative, monotone and submodular function $f$ defined on the subsets of
$M$. Here $M$ is the \emph{ground set}, and $f$ is the \emph{rank function}.
A polymatroid is \emph{integer} if the rank function takes integer values
only, and it is a \emph{matroid} if it is integer, and all singletons have
value either zero or one. For a throughout treatment of matroids see
\cite{oxley}. Polymatroids were introduced by Edmonds \cite{edmonds},
relevant results on polymatroids can be found, e.g., in 
\cite{one-adhesive,lovasz,fmadhe}.

A polymatroid can be considered as ``fractional matroid''. Relaxing a
combinatorial notion to its fractional version allows different techniques
to apply. This not different in case of polymatroids. While tools handling
matroids are mainly combinatorial, polymatroids have a nice geometrical
interpretation allowing geometrical (and continuity) reasoning.

Following the usual practice, ground sets and their subsets are denoted by
capital letters, their elements by lower case letters. The union sign is
frequently omitted as well as the curly brackets around singletons, thus
$abA$ denotes the set $\{a,b\}\cup A$. The set difference has
lower priority than union or intersection, thus $aA{-}b\cap B$ denotes 
$(\{a\}\cup A){-}(\{b\}\cap B)$. For a function defined on subsets of a set
the usual information-theoretical abbreviations are used. Most notably, we
write $f(I|K)$ for $f(IK)-f(K)$.

\smallskip

A (discrete) \emph{measure $\mu$} on $M$  is an additive function on the
subsets of $M$ with $\mu(\emptyset)=0$. As $M$ is finite, the measure is 
determined by its value on the singletons as
$$
    \mu(A) = {\textstyle\sum} \{ \mu(a):\,a\in A \}.
$$
If $\M=(f,M)$ is a polymatroid, then the measure $\mu_\M$, or just
$\mu_f$, is the one defined by the singleton ranks.
Submodularity implies $\mu_f(A)\ge f(A)$ for all $A\subseteq M$.

\smallskip

A collection of subsets is a \emph{lattice $\LL$} if any two elements
$Z_1,Z_2\in\LL$ have, in $\LL$, a least upper bound -- their \emph{join} --,
and a greatest lower bound -- their \emph{meet} --, using the standard
subset relation as the ordering. The notation for join and meet are $Z_1\lor
Z_2$ and $Z_1\land Z_2$, respectively. We write $Z_1<Z_2$ if $Z_1$ is
strictly below $Z_2$, which is the same as $Z_1\subset Z_2$. The lattice
elements $Z_1$ and $Z_2$ are \emph{incomparable} if they are
different and neither $Z_1<Z_2$ nor $Z_2<Z_1$ holds. As $\LL$ is finite, it
has a smallest element (the meet of all $Z\in\LL$), denoted by
$O_\LL$, and a largest element (the join of all $Z\in\LL$), denoted by
$I_\LL$. The pair $(\lambda,\LL)$ is a \emph{ranked lattice} if the rank
function $\lambda$ assigns non-negative real values to lattice elements.
$\lambda$ is \emph{pointed} if $\lambda(O_\LL)=0$, and \emph{monotone} if
$Z_1\le Z_2$ implies $\lambda(Z_1)\le \lambda(Z_2)$.

\subsection{Flats, cyclic flats}

In this section $\M=(f,M)$ is a fixed polymatroid. The element $a\in M$ is a
\emph{loop} if it has rank zero: $f(a)=0$, in which case $f(aA)=f(A)$ for all
$A\subseteq M$. The element $a\in M$ is a \emph{coloop} or \emph{isthmus} if
$f(a|M{-}a)=f(a)$, which means $f(aA)=f(A)+f(a)$ for every $A\subset M$ not
containing $a$. Both loops and coloops add trivial structural properties
only, thus very frequently the polymatroid is, or can be, assumed to have 
no loops and no coloops.

The subset $F\subseteq M$ is a \emph{flat} if proper extensions of $F$ have
strictly larger ranks. The intersection of flats is a flat, and the
\emph{closure} of $A\subseteq M$, denoted by $\cl_\M(A)$, or simply by
$\cl(A)$ when $\M$ is clear from the context, is the smallest flat
containing $A$. The ground set $M$ is always a flat. Flats of a polymatroid
form a lattice where the meet is the intersection, and the join is the
closure of the union. Flats immediately under the maximal flat $M$ are
called \emph{hyperplanes}. The minimal flat is the collection of all loops,
which can be the empty set if $\M$ has no loops.

The flat $C\subseteq M$ is \emph{cyclic} if for all $i\in C$ either $i$ is a
loop, or $f(i|C{-}i)<f(i)$, see \cite{pcyclic,distr}. In particular, the
minimal flat (containing only loops) is cyclic. When $\M$ is a matroid, this
definition of cyclic flats is equivalent to the original one, namely that
$C$ is the union of circuits (minimal connected sets), see
\cite{cyclic-flats,oxley}.

Similar to the matroid case, cyclic flats form a lattice. The proof relies
on the following structural property of cyclic flats.

\begin{lemma}\label{lemma:cyclic-flat}
Every flat $F$ contains a unique maximal cyclic flat $C\subseteq F$;
moreover for all $C\subseteq A \subseteq F$ we have
$$
    f(A) = f(C)+\mu_f(A{-}C).
$$
\end{lemma}
\begin{proof}
First we show that $F$ contains a maximal cyclic flat $C$ with the given
property, then we show that the maximal cyclic flat inside $F$ is unique.

Start with $F_1=F$, and suppose we have defined $F_j$ for some $j\ge 1$. If
there is an element $x_j\in F_j$ such that $f(x_j)>0$ and
$f(F_j)-f(F_j{-}x_j) = f(x_j)$, then let $F_{j+1}=F_j{-}x_j$, otherwise
stop. Submodularity gives that for each $i<j$
$f(x_iF_j)=f(x_i)+f(F_j)>f(F_j)$, thus $F_j$ is a flat (as $F_j$ has smaller
rank than $x_iF_j$ for every $x_i\in F{-}F_j$), and the last $F_j=C$ is
cyclic. As no proper extension of $C$ inside $F$ is cyclic, it is a maximal
cyclic flat in $F$ with the claimed property.

Second, suppose $C\subseteq F_1=F$ is a maximal cyclic flat. If $i\in C$ is
not a loop then $f(i|F_1{-}i)\le f(i|C{-}i)< f(i)$, consequently $C$ is a
subset of $F_2=\{i\in F_1: i$ is a loop or $f(i|F_1{-}i)<f(i)\}$. Define
similarly the sets $F_1\supseteq F_2\supseteq F_3 \supseteq\cdots\supseteq
C$. It is clear that each $F_j$ is a flat, and when $F_j=F_{j+1}$ then it is
cyclic. As it contains $C$, it must equal $C$.
\end{proof}

\begin{corollary}\label{corr:unique}
The rank of cyclic flats and singletons determine the polymatroid.
\end{corollary}
\begin{proof}
This is so as the rank of $A$ is the minimum of $f(C)+\mu_f(A{-}C)$ as 
$C$ runs over all cyclic flats. Indeed,
$f(A)\le f(C)+f(A{-}C)\le f(C)+\mu_f(A{-}C)$ by submodularity, thus it is
enough to show that for some cyclic flat $C$ equality holds. Let $F=\cl(A)$ and
$C\subseteq F$ be the maximal cyclic flat in $F$. Then $f(F)=f(A)=f(AC)$,
and by Lemma \ref{lemma:cyclic-flat},
$$
    f(AC)=f(C)+\mu_f(AC{-}C)=f(C)+\mu_f(A{-}C),
$$
as required.
\end{proof}

\begin{claim}\label{claim:cyclic-flats-is-lattice}
Cyclic flats of a polymatroid form a lattice.
\end{claim}
\begin{proof}
Let $C_1$ and $C_2$ be cyclic flats. Then $C_1\cap C_2$ is a flat which
contains a unique maximal cyclic flat by Lemma \ref{lemma:cyclic-flat} above. This
is the largest cyclic flat below $C_1$ and $C_2$. The smallest upper bound
of $C_1$ and $C_2$ is $C=\cl(C_1\cup C_2)$. Indeed, this is a flat, and we
claim that it is also cyclic. If $i\in C{-}C_1C_2$ then $i$ is not a loop 
(as loops are in $C_1\cap C_2$), 
$f(iC_1C_2)=f(C)$, thus $f(i|C{-}i)=0<f(i)$. If, say, $i\in C_1$ and
$f(i)>0$ then $f(i|C{-}i)\le f(i|C_1{-}i)<f(i)$ proving that $C$ is cyclic
indeed.
\end{proof}

\subsection{Convolution}\label{subsec:convolution}

Let $(\lambda,\LL)$ be a ranked lattice and $\mu$ be a (discrete) measure
both defined on subsets of $M$. The \emph{convolution} of the ranked lattice and the
measure, denoted by $\lambda*\mu$, assigns a non-negative value to each
subset of $M$ as follows:
\begin{equation}\label{eq:convolution}
   \lambda*\mu \,:\, A\mapsto \min\,\{ \lambda(Z)+\mu(A{-}Z)\,:Z\in \LL \}.
\end{equation}
In the sequel we write $r$ instead of $\lambda*\mu$. When $\LL$ contains all
subsets of $M$ and $\lambda$ is the rank function of a polymatroid, then
(\ref{eq:convolution}) is equivalent to the usual convolution of two
polymatroids, see \cite{lovasz,matus-csirmaz}.

\section{Characterizing cyclic flats}\label{sec:characterizing}

\subsection{Conditions}\label{subsec:conditions}

Convolution will be used to recover a polymatroid from the lattice of its 
cyclic flats. Different conditions on the ranked lattice and the measure ensure
different properties of the convolution. Rather than listing them repeatedly,
we specify them here. In what follows, $(\lambda,\LL)$ is a ranked lattice and 
$\mu$ is a measure, both defined on subsets of the same set.

\begin{itemize}\setlength\itemsep{0pt}%
\setlength\belowdisplayskip{2pt plus 1 pt minus 1pt}%
\setlength\abovedisplayskip{5pt plus 1pt minus 1pt}%
\item[\C1\CF] $\lambda(O_\LL)=0$, i.e., $\lambda$ is pointed.
\item[\C2\CF] for comparable lattice elements $Z_1 \le Z_2\in\LL$,
$$ 
    0\le \lambda(Z_2)-\lambda(Z_1)\le \mu(Z_2{-}Z_1).
$$
\item[\C*] for different comparable lattice elements $Z_1 < Z_2\in\LL$,
$$
    0<\lambda(Z_2)-\lambda(Z_1)<\mu(Z_2{-}Z_1).
$$
\item[\C3\CF] for any two lattice elements $Z_1,Z_2\in\LL$,
$$
  \lambda(Z_1)+\lambda(Z_2)\ge \lambda(Z_1\lor Z_2)+\lambda(Z_2\land Z_2)
   +\mu(Z_1\cap Z_2{-}Z_1\land Z_2).
$$
\item[\C4\CF] if $a\in Z\in\LL$, then $\mu(a)\le \lambda(Z)$.
\item[\C5\CF] a) $\lambda(Z)>0$ for $Z\not=O_\LL$;
b) $\mu(a)>0$ for $a\notin O_\LL$.
\end{itemize}
Condition \C2 is monotonicity, \C* is strict monotonicity.
\C3 is submodularity corrected for the the difference 
between $Z_1\cap Z_2$ and $Z_1\land Z_2$ (the meet is always a subset of
the intersection). Some remarks are due.

\begin{remark}\label{remark:1}
\C3 trivially holds when $Z_1$ and $Z_2$ are comparable, thus it is
enough to require it to hold for incomparable $Z_1$ and $Z_2$.
\end{remark}

\begin{remark}\label{remark:2}
\C1 and \C4 implies $\mu(a)=0$ for $a\in O_\LL$ as $0\le\mu(a)\le
\lambda(O_\LL)=0$. \C5 a) trivially follows from \C*.
\end{remark}

\begin{remark}\label{remark:3}
In conditions \C2, \C* and \C3 the subset for which $\mu$ is applied is
always disjoint from $O_\LL$, thus the values $\mu(a)$ for $a\in O_\LL$ are
irrelevant.
\end{remark}

\begin{remark}\label{remark:4}
Conditions \C2, \C* and \C3 are homogeneous in $\lambda$. Thus they hold
if and only if they hold for the pointed rank function
$\lambda(Z)-\lambda(O_\LL)$.
\end{remark}

\begin{remark}\label{remark:5}
Suppose $\mu(a)=0$ for $a\in O_\LL$ and $\mu(a)=1$ otherwise, and that
$\lambda$ is integer valued. Then \C* implies both \C4 and \C5. Consequently
``\C1 and \C* and \C3 and \C4 and \C5'' is equivalent to ``\C1 and \C* and
\C3'', which is the same as the list of axioms in \cite[Theorem
3.2]{cyclic-flats}, as by Remark \ref{remark:3} in this case the measure
$\mu$ can be replaced by the cardinality function.
\end{remark}

\subsection{Cyclic flats of polymatroids}\label{subsec:characterization}

For each polymatroid $\M$ we define a pair of a ranked lattice and a measure
as follows: the ranked lattice is collection of cyclic flats endowed with
the polymatroid rank, and the measure is $\mu=\mu_\M$ generated by the rank
of singletons.

\begin{theorem}\label{thm:main}
The pair of the ranked lattice $(\lambda,\LL)$ and the measure $\mu$ is
defined from a polymatroid if and only if they satisfy \C1, \C*, \C3, \C4
and \C5. This polymatroid is uniquely defined, and is integer
if and only if $\lambda$ and $\mu$ are integer valued.
\end{theorem}

By Remark \ref{remark:5}, the result in \cite{cyclic-flats} characterizing
the lattice of cyclic flats of matroids follows immediately from this
theorem. The proof proceeds in two stages. The easy part is Claim
\ref{claim:conditions-necessary}, which shows that the lattice and the
measure coming from a polymatroid satisfy the conditions. The converse
follows from the fact that the convolution $\lambda*\mu$ recovers a
polymatroid which defines $\lambda$ and $\mu$. This is proved in Claim
\ref{claim:conditions-sufficient} using a series of claims and lemmas from
Section \ref{sec:proof} highlighting the role of each condition. The uniqueness
follows from Corollary \ref{corr:unique}: $\lambda$ and $\mu$ together 
determine the polymatroid rank function.

\begin{claim}\label{claim:conditions-necessary}
Let $\LL$ be the lattice of cyclic flats of the polymatroid $\M=(f,M)$.
Define $\lambda(Z)=f(Z)$ for $Z\in\LL$.
Properties \C1, \C*, \C3, \C4, \C5 hold for $(\lambda,\LL)$ and $\mu=\mu_\M$.
\end{claim}
\begin{proof}
The minimal cyclic flat $O_\LL$ is the set of loops, i.e., elements with
rank zero. This proves \C1 and \C5; \C4 is a simple consequence of submodularity.

If $Z_1<Z_2$ then $f(Z_1)<f(Z_2)$ as $Z_1$ is a flat and $Z_2$ is a proper 
extension of $Z_1$. Suppose by contradiction that
$f(Z_2)-f(Z_1)=\mu(Z_2{-}Z_2)$, and let $a\in Z_2{-}Z_1$.
Then $f(Z_1)+\mu(Z_2{-}Z_1a)\ge f(Z_2{-}a)$ by submodularity, and
$$
   f(a)+f(Z_2{-}a)\ge f(Z_2)=f(Z_1)+\mu(Z_2{-}Z_1)\ge f(a)+f(Z_2{-}a),
$$
which means $f(Z_2)=f(Z_2{-}a)+f(a)$, thus $Z_2$ is not cyclic, proving \C*.

Finally, \C3 follows from Lemma \ref{lemma:cyclic-flat}, since
$Z_1\lor Z_2=\cl(Z_1\cup Z_2)$, the ranks are equal $f(Z_1\lor
Z_2)=f(Z_1\cup Z_2)$. Let $F=Z_1\cap Z_2$; it is a flat, and $C=Z_1\land Z_2$
is the maximal cyclic flat contained in $F$. By Lemma
\ref{lemma:cyclic-flat}
$$
   f(F)=f(Z_1\cap Z_2) = f(C) + \mu_f(F{-}C).
$$
Combining these with the submodularity
$$
   f(Z_1)+f(Z_2) \ge f(Z_1\cup Z_2)+f(Z_1\cap Z_2)
$$
we get the inequality in \C3.
\end{proof}

\begin{claim}\label{claim:conditions-sufficient}
Suppose the ranked lattice $(\lambda,\LL)$ and the measure $\mu$ satisfy
conditions \C1, \C*, \C3, \C4 and \C5. The convolution $\lambda*\mu$
recovers a polymatroid $\M$ such that
$(\lambda,\LL)$ is the lattice of cyclic flats of $\M$ endowed
with the polymatroid rank, and $\mu=\mu_\M$.
\end{claim}
\begin{proof}
By Theorem \ref{thm:convolution} the convolution is a polymatroid.
$\mu=\mu_\M$ follows from Lemma \ref{lemma:zero-min-LL} c). By Claims
\ref{claim:conv-to-LL} and \ref{claim:LL-to-conv} elements of the lattice
$\LL$ are precisely the cyclic flats of $\M$. Finally, Claim
\ref{claim:r-lambda-equal} says that lattice and polymatroid ranks are
equal.
\end{proof}

\section{Convolution properties}\label{sec:proof}

In this section $(\lambda,\LL)$ is a fixed ranked lattice and $\mu$ is a
measure both defined on subsets of $M$. The convolution function
defined on subsets of $M$ will be denoted by $r$:
\begin{equation}\label{eq:conv-again}
    r: A\mapsto \min \{ \lambda(Z)+\mu(A{-}Z\,:\, Z\in\LL \}.
\end{equation}

\begin{theorem}\label{thm:convolution}
If \C3 holds, then $(r,M)$ is a polymatroid.
\end{theorem}
\begin{proof}
First observe that for arbitrary subsets $A,B,Z_A,Z_B$ of $M$ we have
\begin{align}\label{eq:basic-ineq}
  &\mu(A{-}Z_A)+\mu(B{-}Z_B)\ge {} \\
  &~~~~\mu(A\cap B{-}Z_A\cap Z_B) + \mu(A\cup B{-}Z_A\cup Z_B)\nonumber
\end{align}
Indeed, if $i\in A\cap B{-}Z_A\cap Z_B$ then $i$ is in both $A$ and $B$ and
not in either $Z_A$ or $Z_B$, thus it is in either $A{-}Z_A$ or
$B{-}Z_B$. If $i\in A\cup B{-}Z_A\cup Z_B)$, then $i$ is in neither $Z_A$
nor $Z_B$, thus, again, it is in either $A{-}Z_A$ or $B{-}Z_B$. Finally, if
$i$ is in both sets on the right hand side of (\ref{eq:basic-ineq}), then 
$i$ is in both $A$ and $B$, and not in $Z_A$ neither in $Z_B$, thus $i$ is in 
both sets on the left hand side.

\smallskip

The convolution $(r,M)$ is a polymatroid if $r$ is non-negative, monotone,
and submodular. Non-negativity is clear from the definition
(\ref{eq:conv-again}) as both $\lambda$ and $\mu$ are non-negative. Let $A$
and $B$ be subsets of $M$; $r(A)=\lambda(Z_A)+\mu(Z_A{-}A)$ and
$r(B)=\lambda(Z_B)+\mu(Z_B{-}B)$. If $A\subseteq B$ then
$$
   r(A)\le \lambda(Z_B)+\mu(A{-}Z_B)\le\lambda(Z_B)+\mu(B{-}Z_B)=r(B),
$$
showing monotonicity. To check submodularity we use $Z_A\land Z_B$ and
$Z_A\lor Z_B$ to estimate $r(A\cap B)$ and $r(A\cup B)$, respectively, as 
follows:
\begin{align*}
   r(A\cap B) &\le \lambda(Z_A\land Z_B)+\mu(A\cap B{-}Z_A\land Z_B), \\
   r(A\cup B) &\le \lambda(Z_A\lor Z_B)+\mu(A\cup B{-}Z_A\lor Z_B).R
\end{align*}
Using condition \C3, the submodularity $r(A)+r(B)\ge r(A\cup B)+r(A\cap B)$ 
follows if
\begin{align*}
   &\mu(A{-}Z_A)+\mu(B{-}Z_B)+\mu(Z_A\cap Z_B{-}Z_A\land Z_B) \ge{} \\
   &~~~~~~~~~~~~\ge \mu(A\cap B{-}Z_A\land Z_B)+\mu(A\cup B{-}Z_A\lor Z_B).
\end{align*}
As 
$$
   \mu(A\cap B{-}Z_A\land Z_B)\le \mu(A\cap B{-}Z_A\cap Z_B)+
     \mu(Z_A\cap Z_B{-}Z_A\land Z_B)
$$
(the right hand side is a disjoint union), and
$$
  \mu(A\cup B{-}Z_A\lor Z_B) \le \mu(A\cup B{-} Z_A\cup Z_B)
$$
as $Z_A\cup Z_B$ is a subset of $Z_A\lor Z_B$,
(\ref{eq:basic-ineq}) gives the required inequality.
\end{proof}

\begin{lemma}\label{lemma:pre-flat}
Suppose $a\notin A$ and $A$ is disjoint from $Z\in\LL$. If
$r(aAZ)=\lambda(Z)+\mu(aA)$, then $r(aAZ)=r(AZ)+\mu(a)$.
\end{lemma}
\begin{proof}
We show that $r(AZ)=\lambda(Z)+\mu(A)$, from here the claim follows. Let
$r(AZ)=\lambda(Z')+\mu(AZ{-}Z')$ for some $Z'\in\LL$. Then $r(aAZ)$ is
minimal, thus
$$
 \lambda(Z)+\mu(aA) = r(aAZ)\le \lambda(Z')+\mu(aAZ{-}Z').
$$
Similarly, $r(AZ)$ is minimal, thus
\begin{equation}\label{eq:lemma-4}
   \lambda(Z')+\mu(AZ{-}Z') = r(AZ) \le 
   \lambda(Z)+\mu(AZ{-}Z) = \lambda(Z)+\mu(A)
\end{equation}
as $AZ{-}Z=A$ by  assumption. Combining them we get
$$
    \mu(aA)+\mu(AZ{-}Z') \le \mu(aAZ{-}Z')+\mu(A),
$$
which holds only if they are equal.
Consequently we have equality in
(\ref{eq:lemma-4}) as was required.
\end{proof}

\begin{lemma}\label{lemma:zero-min-LL}
{\upshape a)} \C1 implies $r(a)=0$ for $a\in O_\LL$, and $r(a)\le \mu(a)$
otherwise.
{\upshape b)} \C4 implies $r(a)\ge\mu(a)$ for all $a\in M$.
{\upshape c)} \C5 implies $r(a)>0$ for $a\notin O_\LL$.
\end{lemma}
\begin{proof}
Immediate from the conditions and from the fact that $r(a)$ is the minimum
of $\lambda(Z)+\mu(a{-}Z)$ as $Z$ runs over $\LL$.
\end{proof}

\begin{lemma}\label{lemma:a-or-z}
{\upshape a)} Assume \C2 and \C3.
For every pair of lattice elements $A,Z\in\LL$ we have
$$
   \lambda(A\lor Z)\le \lambda(Z)+\mu(A{-}Z).
$$
{\upshape b)} Assume \C* and \C3.
For every pair $A,Z\in\LL$ such that $A$ is not below $Z$  the
above inequality is strict.
\end{lemma}
\begin{proof}
If $A\le Z$ then $A\lor Z=Z$, thus the two sides are equal. When $Z<A$ then
the inequality (strict inequality) follows from condition \C2 (condition
\C*, respectively). Finally, if $A$ and $Z$ are incomparable, then apply
\C2 (or \C*) for $A\land Z$ and $Z$,
and \C3 for $A$ and $Z$ to get
\begin{align*}
  \lambda(A) &\le \lambda(A\land Z) + \mu(A{-}A\land Z), \\[2pt]
 \lambda(A\land Z)+\lambda(A\lor Z) &\le
     \lambda(A)+\lambda(Z)-\mu(A\cap Z{-}A\land Z).
\end{align*}
Their sum is the claimed inequality. When using \C*, the first
inequality is strict, thus the sum is strict as well.
\end{proof}

\begin{claim}\label{claim:r-lambda-equal}
Assume \C2 and \C3.  $r(A)=\lambda(A)$ for every $A\in\LL$.
\end{claim}
\begin{proof}
Using $Z=A$ in the definition $r(A)=\min\{ \lambda(Z)+\mu(A{-}Z)\}$ gives
$r(A)\le\lambda(A)$. To show the converse, condition \C2 gives $\lambda(A)\le
\lambda(A\lor Z)$, and by Lemma \ref{lemma:a-or-z} a),
$$
    \lambda(A) \le \lambda(A\lor Z)\le \lambda(Z)+\mu(A{-}Z)
$$
for every $Z\in\LL$, thus $\lambda(A)\le r(A)$.
\end{proof}

\begin{claim}\label{claim:conv-to-LL}
Assume \C1, \C3 and \C5. Every cyclic flat of
the convolution is an element of the lattice $\LL$.
\end{claim}
\begin{proof}
\C3 implies that the convolution $(r,M)$ is a polymatroid.
Suppose $F\subseteq M$ is a flat in it, and $r(F)=\lambda(Z)+\mu(F{-}Z)$
for some $Z\in\LL$. Now
$$
   r(FZ)\ge
   r(F)=\lambda(Z)+\mu(F{-}Z) = \lambda(Z)+\mu(FZ{-}Z) \ge r(FZ),
$$
consequently $r(F)=r(FZ)$. As $F$ is a flat, $Z\subseteq F$. Suppose $F{-}Z$
is not empty, let $F{-}Z=aA$ with $a\notin A$. As $r(aAZ)=\lambda(Z) +
\mu(aA)$, Lemma \ref{lemma:pre-flat} gives $r(F)-r(F{-}a)=\mu(a)$. By Lemma
\ref{lemma:zero-min-LL} a) and c) $\mu(a)\ge r(a)>0$ (and then $\mu(a)$ and
$r(a)$ must be equal), thus $F$ is not cyclic.
\end{proof}

\begin{claim}\label{claim:LL-to-conv}
Assume \C*, \C3, \C4, \C5 b). Every $Z\in\LL$ is a cyclic flat in $(r,M)$.
\end{claim}
\begin{proof}
By Claim \ref{claim:r-lambda-equal}, $r(Z)=\lambda(Z)$ for all lattice
elements. First we check that $Z\in\LL$ is a flat. Let $a\notin Z$, we want
to show that $aZ$ has larger rank than $Z$. As $a\notin O_\LL$, condition
\C5 b) says $\mu(a)>0$. Suppose
$r(aZ)=\lambda(Z')+\mu(aZ{-}Z')$. If $Z<Z'$ then $\lambda(Z)<\lambda(Z')$ by
\C*. If $Z=Z'$ then $r(aZ)=\lambda(Z)+\mu(a)>\lambda(Z)$. Otherwise $Z$ is
not below $Z'$, and then Lemma \ref{lemma:a-or-z} b) gives
$$
   \lambda(Z)\le \lambda(Z\lor Z')<\lambda(Z')+\mu(Z{-}Z')\le r(aZ),
$$
as required.
To show that $Z\in\LL$ is cyclic, let $a\in Z$ and suppose by
contradiction that $r(a)>0$ and $r(Z{-}a)=r(Z)-r(a)$. Let
$$
    r(Z{-}a) = \lambda(Z')+\mu(Z{-}aZ')=\lambda(Z)-r(a).
$$
As $\lambda(Z')<\lambda(Z)$, $Z\le Z'$ is impossible by \C*. Thus $Z$ is
not below $Z'$, and Lemma \ref{lemma:a-or-z} b) gives
\begin{align*}
  \lambda(Z) \le \lambda(Z\lor Z') &< \lambda(Z')+\mu(Z{-}Z') \\
   & \le \lambda(Z')+\mu(Z{-}aZ')+\mu(a) \\
   &=\lambda(Z)-r(a)+\mu(a).
\end{align*}
But this is impossible as by Lemma \ref{lemma:zero-min-LL} b),
$r(a)\ge\mu(a)$.
\end{proof}

\section{An example}

An illustrative example for using convolution
is a proof of Helgason's theorem \cite{helgason}
saying that integer polymatroids are factors of matroids. For other
examples see \cite{sticky-poly}. Let
$\M=(f,M)$ be an integer polymatroid. For $i\in M$ find disjoint sets $M_i$
such that $i\in M_i$ and $M_i$ has exactly $\max\{ 1, f(i) \}$ elements.
The lattice $\LL$ consists of subsets $Z\subseteq N=\bigcup_i M_i$
which have the property
$$
   \mbox{if $M_i\cap Z\not=\emptyset$, then $M_i\subseteq Z$}.
$$
The meet and join are the union and intersection, and $O_\LL$ is the empty
set. Define the rank $\lambda$ as
$$
    \lambda: Z \mapsto f(Z\cap M).
$$
The measure is the expected one: for $a\in M_i$ let $\mu(a)=\min\{ 1,f(i) \}$.
It is a routine to check that conditions \C1, \C2, \C3 and \C4 hold. Let
$\N=(r,N)$ be the convolution of the ranked lattice and the measure. 
It is clearly integer, and by Theorem \ref{thm:convolution} it is a
polymatroid. By Lemma \ref{lemma:zero-min-LL} $r(a)=\mu(a)$, consequently
the rank of singletons is either zero or one, which shows thus $\N$ is a 
matroid. Finally, Claim
\ref{claim:r-lambda-equal} says $r(Z)=\lambda(Z)$ for all $Z\in\LL$,
therefore $\M$ is a factor of $\N$ as required.

\medskip

Helgason's theorem is a special case of a more general statement. In the
construction above each singleton $i\in M$ with rank $f(i)\ge 2$ is
replaced by the free matroid of rank $f(i)$. Actually any matroid with this
rank can be used which is an immediate consequence of Theorem
\ref{thm:infiltrate}. Let us proceed with some definitions.

The \emph{convolution} of two ranked lattices $(\lambda_1,\LL_1)$ and
$(\lambda_2,\LL_2)$ is the function on subsets of $I_{\LL_1}\cup
I_{\LL_2}$ defined as
$$
   A \mapsto \min\nolimits_{Z_1,Z_2} \big\{
    \lambda_!(Z_1)+\lambda_2(Z_2) : A \subseteq Z_1\cup Z_2
   \big\}
$$
as $Z_1$ runs over elements of $\LL_1$ and $Z_2$ runs over elements of
$\LL_2$. If $\LL_2$ is the complete subset lattice and $\lambda_2$ is a
measure, then this formula is equivalent to (\ref{eq:convolution}) from
Section \ref{subsec:convolution}. In the very special case of Theorem
\ref{thm:infiltrate} the convolution of two ranked lattices defines a
polymatroid. It is an interesting problem under which general and useful
conditions this remains true.

Let $P,M$ be disjoint sets, and let $c\notin M\cup P$. $\M=(f,Mc)$ and 
$\N=(g,P)$ are polymatroids with $f(c)=g(P)$. The polymatroid  $(r,MP)$
\emph{infiltrates $P$ under $c$} if for all subsets $A\subseteq M$ and 
$B\subseteq P$ we have
\begin{align}\label{eq:infiltrate}
   r(A) &= f(A),~~~~  r(B)=g(B), \\
   r(AP) &= f(Ac),\nonumber
\end{align}
that is, $r$ extends both $f\restr M$ and $g$, and inserts $P$ in place
of $c$.

\begin{theorem}\label{thm:infiltrate}
For each $\M=(f,Mc)$ and $\N=(g,P)$ with $f(c)=g(P)$ one can
infiltrate $P$ under $c$.
\end{theorem}
\begin{proof}
Define two ranked lattices on subsets of $MP$ as follows:
\begin{align*}
  \LL_1 &= \{ Z_1\subseteq MP : \,\mbox{ $Z_1\cap P=\emptyset$, or 
                                        $Z_1\cap P=P$} \}, \\
  \LL_2 &= \{ Z_2\subseteq MP: \,\mbox{ $Z_2\cap M=\emptyset$} \}
\end{align*}
with rank functions
\begin{align*}
  \lambda_1(Z_1) &= \begin{cases}
                  f(Z_1\cap M) & \mbox{if $Z_1\cap P=\emptyset$},\\
                  f(c(Z_1\cap M)) & \mbox{if $Z_1\cap P=P$};
                  \end{cases} \\[2pt]
  \lambda_2(Z_2) &= g(Z_2\cap P).
\end{align*}
Let the convolution of the ranked lattices be $(r,MP)$. We claim that this
is the required extension. As both $\lambda_1$ and $\lambda_2$ are monotone,
the minimum is taken when $Z_1$ and $Z_2$ is the smallest possible.
Consequently for every $A\subseteq MP$,
\begin{equation}\label{eq:infiltrate2}
   r(A)=\min\big\{ f(A\cap M)+g(A\cap P), f(c(A\cap M)) \big\}.
\end{equation}
Using that $f(c)=g(P)$, conditions in (\ref{eq:infiltrate}) follow easily. Thus
one has to check only that $(r,MP)$ is a polymatroid. Non-negativity and
monotonicity is clear, and submodularity can be shown by a case by case
checking depending on which terms in (\ref{eq:infiltrate2}) provide the 
smaller value.
\end{proof}

\section*{Acknowledgment}

The author would like to thank the generous support of 
the Institute of Information Theory and Automation of the CAS, Prague.
The research reported in this paper was supported by GACR project
number 19-04579S, and partially by the Lend\"ulet program of the HAS.

\end{document}